\documentclass{amsart}
\usepackage{layout}
\usepackage[utf8]{inputenc}
\usepackage{amsmath}
\usepackage{textcase}
\usepackage{amsfonts}
\usepackage{amssymb}
\usepackage{amsthm}
\usepackage{bbm}
\usepackage{graphicx}

\renewcommand{\sqrt}[1]{\left( #1 \right)^\frac12}

\newcommand{\e}[0]{\mathrm{e}}
\newcommand{\V}[0]{U}
\newcommand{\D}[0]{\mathbb{D}}

\newcommand{\E}[0]{\mathbb{E}}
\newcommand{\Z}[0]{\mathbb{Z}}
\newcommand{\T}[0]{\mathbb{T}}
\newcommand{\N}[0]{\mathbb{N}}

\renewcommand{\d}[0]{\mathrm{d}}

\renewcommand{\Re}{\mathrm{Re}\,}

\newcommand{\p}[1]{\left(#1\right)}

\newcommand{\acts}[2]{#1:#2\, \rotatebox[origin=c]{90}{$\circlearrowleft$}}
\newcommand{\id}{\mathrm{id}}

\newcommand{\beq}{\begin{equation}}
\newcommand{\eeq}{\end{equation}}
\newcommand{\supp}[0]{\mathrm{supp}\,}
\newcommand{\clspan}[0]{\overline{\mathrm{span}}\,}

\newtheorem{thm}{Theorem}[section]
\newtheorem{lem}[thm]{Lemma}

\newtheorem{deff}[thm]{Definition}
\newtheorem{cor}[thm]{Corollary}
\newtheorem{prop}[thm]{Proposition}
\numberwithin{equation}{section}

\allowdisplaybreaks

\theoremstyle{plain}
\newcounter{thmintr}

\newtheorem{intrthm}[thmintr]{Theorem}

\author{Maciej Rzeszut}
\address{Institute of Mathematics of Polish Academy of Sciences\\
Śniadeckich 8\\
00-656 Warszawa}
\email{mrzeszut@impan.pl}
\author{Michał Wojciechowski}
\address{Institute of Mathematics of Polish Academy of Sciences\\
Śniadeckich 8\\
00-656 Warszawa}
\email{miwoj@impan.pl}
\title{Hoeffding decomposition in $H^1$ spaces}

\begin{document}

\begin{abstract}The well known result of Bourgain and Kwapień states that the projection $P_{\leq m}$ onto the subspace of the Hilbert space  $L^2\left(\Omega^\infty\right)$ spanned by functions dependent on at most $m$ variables is bounded in $L^p$ with norm $\leq c_p^m$ for $1<p<\infty$. We will be concerned with two kinds of endpoint estimates. We prove that $P_{\leq m}$ is bounded on the space $H^1\left(\mathbb{D}^\infty\right)$ of functions in $L^1\left(\mathbb{T}^\infty\right)$ analytic in each variable. We also prove that $P_{\leq 2}$ is bounded on the martingale Hardy space associated with a natural double-indexed filtration and, more generally, we exhibit a multiple indexed martingale Hardy space which contains $H^1\p{\D^\infty}$ as a subspace and $P_{\leq m}$ is bounded on it. \par\end{abstract}
\maketitle
\tableofcontents

\section{Introduction} 
The Rademacher functions $\p{r_i}_{i\in \N}$ generate a well studied subspace of $L^p[0,1]$, which we identify with $L^p\p{\mathbb{Z}_2^\N}$. In particular by Khintchine inequality
\beq \left\|\sum_i c_i r_i\right\|_{L^p}\simeq_p \p{\sum_i \left|c_i\right|^2}^\frac12\eeq
for $0<p<\infty$, and $\overline{\mathrm{span}}\p{r_i:i\in \N}$ is complemented in $L^p$ for $1<p<\infty$ but not for $p=1$. We will index the Walsh system by finite subsets of $\N$, i.e.
\beq w_A=\prod_{i\in A}r_i.\eeq
The number $|A|$ is called the mutiplicity of $w_A$. Analogous problems for Walsh functions of finite multiplicity have been resolved independently by Bonami \cite{bonami} and Kiener \cite{kienerineq}. Namely, the inequality
\beq \label{eq:intrbonkiel2}\left\|\sum_{|A|\leq m} c_A w_A\right\|_{L^p}\simeq_{p,m} \p{\sum_{|A|\leq m} \left|c_A\right|^2}^\frac12\eeq
holds true, and the orthogonal projection $P_m$ onto $\overline{\mathrm{span}}\p{w_A:|A|\leq m}$ is bounded if and only if $1<p<\infty$. Some lower estimates for $P_m$ are also known, see \cite{kienernonineq} or \cite{mullbook} for a detailed discussion on this subject. \par
In \cite{bourgwalsh}, Bourgain generalized these results to a setting in which $\p{\mathbb{Z}_2,\left\{ \emptyset ,\{0\},\{1\},\{0,1\}\right\},\frac{1}{2}\#}$ is replaced with an arbitrary probability space $\p{\Omega,\mathcal{F},\mu}$. To be more precise, let $\p{\Omega^\infty,\mathcal{F}^{\otimes\infty},\mu^{\otimes\infty}}$ be the infinite product space. Any $f\in L^2\p{\Omega^\infty,\mathcal{F}^{\otimes\infty},\mu^{\otimes\infty}}$ can be decomposed in a unique way into a series
\beq f(x)=\sum_m \sum_{i_1<\ldots<i_m}f_{i_1,\ldots,i_m}\p{x_{i_1},\ldots,x_{i_m}}\eeq
where $f_{i_1,\ldots,i_m}\in L^2\p{\Omega^m}$ is mean zero in each of its $m$ arguments. Thus, $P_A$ and $P_m$ defined by 
\beq P_{\left\{i_1,\ldots,i_m\right\}}f(x)= f_{i_1,\ldots,i_m}\p{x_{i_1},\ldots,x_{i_m}},\quad P_m=\sum_{|A|=m}P_A\eeq
are mutually orthogonal orthogonal projections. In the case of $\Omega=\Z_2$, the image of $P_A$ is just the one-dimensional space spanned by $w_A$, so the above definition of $P_m$ coincides for with the projection onto Walsh functions of multiplicity $m$. In \cite{bourgwalsh} Bourgain proved that for $1\leq p<\infty$,
\beq \label{eq:intrbourgsqfn}\left\| \sum_{|A|\leq m}P_A f\right\|_{L^p}\simeq_{p,m}\left\| \p{\sum_{|A|\leq m}\left|P_A f\right|^2}^\frac12\right\|_{L^p},\eeq
which is a direct generalization of \eqref{eq:intrbonkiel2}. Moreover, he proved that $P_m$ is bounded on $L^p$ if and only if $1<p<\infty$, with norm smaller than $c_p^m$ where $c_p\lesssim \frac{ \hat{p}^\frac52}{\log \hat{p}}$ and $\hat{p}=p\vee\frac{p}{p-1}$. \par
It turns out that the projections $P_m$ have a well established probabilistic interpretation. In \cite{kwap}, Kwapień connected them to the notion of Hoeffding decomposition, which originated from Hoeffding's work \cite{wassily}. More precisely, elements of the image of $P_m$ are what is called generalized canonical $U$-statistics and the decomposition $f=\sum_m P_m f$ plays a crucial role in the proofs of many theorems concerning $U$-statistics. For more information, we refer the reader to \cite{leeustat}. Kwapień provided a shorter proof of Bourgain's result about boundedness of $P_m$, with a better constant $c_p\lesssim \frac{\hat{p}}{\log\hat{p}}$. \par
Let us decribe the main results of this paper, which give certain endpoint estimates for $P_m$. One of them (Theorem \ref{PmonH1Dinfty} in the text) is obtained by restricting the domain of $P_m$. For exact definition of $H^1_{\mathrm{all}}$, see Section \ref{secoverview}. 
\begin{intrthm}\label{intrpmh1all}$P_m$ is bounded on the subspace $H^1_{\mathrm{all}}\p{\T^\infty}$ of $L^1\p{\T^\infty}$ consisting of functions analytic in each variable. \end{intrthm}
We also find a norm stronger than $L^1$ and weaker than $L^p$ ($p>1$), in which $P_m$ is bounded. The detailed construction is described in Section 5.2.
\begin{intrthm}For any $m\in \N$, there is a partition of the family of finite subsets of $\N$ into $\dot\bigcup_{i\in \mathcal{I}}\mathcal{A}_i$ such that the norm
\beq \|f\|:= \E\p{\sum_i \left|\sum_{A\in\mathcal{A}_i}P_A f\right|^2}^\frac12 \eeq
is between $L^1$ and all $L^p$ ($p>1$) and $P_m$ is bounded in this norm. 
\end{intrthm}
It is worth noting that Theorem \ref{intrpmh1all} translates directly to the space $\mathcal{H}^1$ of Dirichlet series, i.e. the closure of polynomials of the form $\sum_{n=1}^N b_n n^{-s}$ in the norm
\beq \left\|\sum_{n=1}^N b_n n^{-s}\right\|_{\mathcal{H}^1}:= \lim_{T\to\infty}\frac{1}{2T} \int_{-T}^T \left|\sum_{n=1}^N b_n n^{-it}\right|\d t.\eeq
The Bohr lift, dating back to \cite{bohr}, is the map
\beq H^1_{\mathrm{all}}\p{\T^\infty}\ni \sum_{k\in \N^{\oplus\N}}a_k \e^{i \langle k,t\rangle}\mapsto \sum_{n\in \N} b_n n^{-s}\in\mathcal{H}^1\eeq
where $a_k=b_n$ for $n$ having the prime number factorization $n=\prod_{j}p_j^{k_j}$. It is an isometry between $H^1_{\mathrm{all}}\p{\T^\infty}$ and the space $\mathcal{H}^1$ of Dirichlet series. Thus, our result is equivalent to the fact that the projection from $\mathcal{H}^1$ onto 
\beq\overline{\mathrm{span}}\left\{n^{-s}:n\text{ has at most }m\text{ prime factors}\right\}\subset \mathcal{H}^1\eeq
 is bounded. For a more detailed exposition of Dirichlet series and their relation to polydisc Hardy spaces, see \cite{queffdirichlet}. \par
The paper is organized as follows. In Section \ref{secoverview} we introduce necessary notation and definitions. In Section \ref{secpmlp}, we provide a new simple proof of the historic $L^p$ boundedness result. The proof of the estimate $\left\|P_m\right\|\leq \p{\e\left\|P_1\right\|}^m$ is done by means of a combinatorial identity expressing $P_m$ in terms of tensor products of $P_1$. In Section \ref{secpmh1}, we show that the same argument carries over with little modification showing boundedness of $P_m$ on $H^1_{\mathrm{all}}\p{\T^\infty}$. In Section \ref{ssdoubleindex}, we define, purely in terms of square functions and not referring to analyticity, a multiple indexed martingale Hardy space $H^1\left[\mathcal{T}_m\right]$ of functions on $\Omega^\infty$ that admits a bounded action of $P_m$. It turns out that if $\Omega=\T$, there is a subspace $H^1_{m\text{ last}}\p{\T^\N}$ of $L^1\p{\T^\infty}$, much bigger than $H^1_{\mathrm{all}}\p{\T^\infty}$, on which the $L^1$ norm is equivalent to $H^1\left[\mathcal{T}_m\right]$ norm. The arguments rely heavily on $L^1$ square function theorem for Hardy martingales and decoupling inequality of Zinn. We present two proofs of the latter in Section \ref{wyrostek}.

\section*{Acknowledgements} The results of this paper are taken from my doctoral thesis \cite{diss}. I am grateful to my advisors: Fedor Nazarov and Michał Wojciechowski for their mentorship and support, especially their unending willingness to discuss my research.

\section{Preliminaries}\label{secoverview}
\textbf{Probability spaces and conditional expectations.} In all of the text, $\left(\Omega,\mathcal{F},\mu\right)$ will be a probability space. We will equip sets of the form $\Omega^I$, where $I$ is an at most countable index set, with the product measure $\mu^{\otimes I}$ defined on $\mathcal{F}^{\otimes I}$. In case we are only concerned with the cardinality of $I$, we will write $\Omega^n$, where $n$ is a natural number or $\infty$. By the natural filtration on $\Omega^\mathbb{N}$ we mean the filtration $\left(\mathcal{F}_n:n=0,1,\ldots\right)$, where $\mathcal{F}_k$ is generated by the coordinate projection $\omega\mapsto \left(\omega_1,\ldots,\omega_k\right)$ and denote $\E_k=\E\left(\cdot\mid\mathcal{F}_k\right)$. In general, for a subset $A$ of the index set, $\mathcal{F}_A$ will be the sigma algebra generated by the coordinate projection $\omega\mapsto\left(\omega_i\right)_{i\in A}$ and $\E_A=\E\left(\cdot\mid\mathcal{F}_A\right)$. In more explicit terms, measurability with respect to $\mathcal{F}_A$ is equivalent to being dependent only on variables with indices belonging to $A$ and the conditional expectation operator $\E_A$ integrates away the dependence on all other variables, so that the formulas 
\beq \E_k f\p{x}= \int_{\Omega^{[k+1,\infty)}} f\p{x_1,\ldots,x_k,y_{k+1},y_{k+2},\ldots}\d \mu^{\otimes [k+1,\infty)}(y),\eeq
\beq \E_A f(x)= \int_{\Omega^{\N\setminus A}}f\p{ x_A,y_{\N\setminus A}}\d \mu^{\otimes \N\setminus A}(y)\eeq
are satisfied (with the convention that sequences indexed by $A$ and $\N\setminus A$ are merged in a natural way into a sequence indexed by $\N$). It will often be convenient to identify a function $f$ defined on $\Omega^A$ with an $\mathcal{F}_A$-measurable function $\Omega^I\ni\omega \mapsto f\p{ \p{\omega_i}_{i\in A}}$. In order to save space, we will often write $\d x$ instead of $\d \mu(x)$ whenever the measure is implied by context.  \par

\textbf{Tensor products.} Let $1\leq p<\infty$. For $f_k\in L^p\p{\Omega_k}$, we will denote by $\bigotimes_{k=1}^n f_k$ the function on $\prod_k \Omega_k$ satisfying
\beq \p{\bigotimes_k f_k}(x)=\prod_k f_k\p{x_k}.\eeq
Because of separation of variables, we have $\left\|\bigotimes_k f_k\right\|_{L^p\p{\prod_k \Omega_k}}= \prod_k \left\|f_k\right\|_{L^p\p{\Omega_k}}$. This way we actually define an injection of the algebraic tensor product $\bigotimes_k L^p\p{\Omega_k}$ into $L^p\p{\prod_k \Omega_k}$, the image of which is dense. \par
Let $X_k$ be subspaces (by a subspace we always mean a closed linear subspace) of $L^p\p{\Omega_k}$. By $\bigotimes_k X_k$ we will denote the subspace of $L^p\p{\prod_k \Omega_k}$ spanned by functions of the form $\bigotimes_k f_k$, where $f_k\in X_k$, and the norm is inherited from $L^p\p{\prod_k \Omega_k}$ (care has to be taken, as $\bigotimes_k X_k$ is not determined solely by $X_k$ as Banach spaces, but rather by the particular way they are embedded in $L^p\p{\Omega_k}$). If $T_k: X_k\to L^p\p{\Omega_k}$ are bounded operators, then we can define an operator $\bigotimes_k T_k: \bigotimes_k X_k\to L^p\p{\prod_k \Omega_k}$ by the formula
\beq \p{\bigotimes_k T_k}\p{\bigotimes_k f_k}= \bigotimes_k T_kf_k,\eeq
and easily check that the property 
\beq \left\|\bigotimes_k T_k:\bigotimes_k X_k\to L^p\p{\prod_k \Omega_k}\right\|\leq \prod_k \left\|T_k:X_k\to L^p\p{\Omega_k}\right\|\eeq
is satisfied. Indeed, $\bigotimes_k T_k= \prod_k \id_{L^p\p{\prod_{j\neq k}\Omega_j}}\otimes T_k$, and any operator of the form $\id\otimes T$ has norm bounded by $\|T\|$, because $\p{\id\otimes T}f\p{\omega_1,\omega_2}= T\p{f\p{\omega_1,\cdot}}\p{\omega_2}$. \par
\textbf{Fourier transform.} Let $\T$ be the interval $[0,2\pi)$ equipped with addition modulo $2\pi$ and normalized Lebesgue measure $\d\mu= \frac{\d x}{2\pi}$. We will be exclusively dealing with Fourier transforms of functions on $\T$ or some power of $\T$. Since the group dual to $\T$ is $\Z$, the dual group to the product $\T^\N$ is the direct sum $\Z^{\oplus\N}$ (i.e., integer-valued sequences that are eventually 0), on which we define the Fourier transform by 
\beq \widehat{f}(n)=\int_{\T^\N} f\p{x}\e^{-i\sum_{k\in \N} n_k x_k}\d \mu^{\otimes \N}(x).\eeq \par
\textbf{Hardy spaces of martingales and analytic functions.} By $\D$ we denote the unit disk in the complex plane. We can identify $\T$ with the unit circle by the map $t\mapsto \e^{it}$. For $N\in\N$, the space $H^1(\D^N)$ is defined as the space of functions analytic in the polydisc $\D^N$ such that the norm
\beq \|F\|_{H^1\p{\D^N}}= \sup_{0<r_1,\ldots,r_n<1} \int_{\T^N} \left|F\p{r_1 \e^{it_1},\ldots,r_n \e^{i t_N}}\right|\frac{\d t}{\p{2\pi}^N}\eeq
is finite. It is well-known \cite{gunstein} that such a function has an a.e. radial limit $f\p{t_1,\ldots,t_n}=\lim_{r\to 1} F\p{r\e^{it_1},\ldots,r\e^{it_n}}$ on the distinguished boundary $\T^N$ and $F$ can be recovered from $f$ by convolution with a Poisson kernel. This sets a one-to-one correspondence between $H^1\p{\D^N}$ and the space 
\beq \label{eq:intrh1all}H^1_{\mathrm{all}}\p{\T^N}= \clspan\left\{\e^{i \langle n,t\rangle}: n_1,\ldots,n_N\geq 0\right\}\subset L^1\left(\T^N\right).\eeq
We also can define $H^1_{\mathrm{all}}\p{\T^\N}$ in the same manner as in \eqref{eq:intrh1all}, but care has to be taken, since these functions are can no longer be extended analytically to $\D^\N$ in general (hence the shorthand $H^1\p{\D^\N}$, which we will sometimes use, is an abuse of notation). 
Later we will use two more $H^1$ spaces, namely $H^1_{\mathrm{last}}\p{\T^\mathbb{N}}$ (also called Hardy martingales) and $H^1_{m\,\mathrm{last}}\p{\T^\mathbb{N}}$, which we will define as follows.
\beq H^1_{\mathrm{last}}\p{\T^\N}= H^1_{1\,\mathrm{last}}\p{\T^\N}= \clspan\left\{\e^{i \langle n,t\rangle}: n_{i_0}>0 \text{ for }i_0=\max\left\{i:n_i\neq 0\right\}\right\}\subset L^1\left(\T^\mathbb{N}\right),\eeq
\beq H^1_{m\,\mathrm{last}}\p{\T^\N}= \clspan\left\{\e^{i \langle n,t\rangle}:m\text{ last of nonzero }n_i\text{'s are }>0\right\}\subset L^1\left(\T^\mathbb{N}\right).\eeq
In the space $H^1_{m\,\mathrm{last}}\p{\T^\N}$ we allow characters of the form $\e^{i\langle n,t\rangle}$, where $\left|\supp n\right|<m$ and $n_j\geq 0$ for all $j$. \par
Now we recall the definition of a martingale Hardy space and some related inequalities. A standard reference in this matter is \cite{garsia}. Let $\left(\mathcal{F}_n\right)_{n=0}^\infty$ be an arbitrary filtration on a probability space $\left(\Omega,\mathcal{F},\mu\right)$, where $\mathcal{F}$ is generated by $\bigcup\mathcal{F}_n$. We denote $\E_k=\E\left(\cdot\mid\mathcal{F}_k\right)$, $\Delta_0=\E_0$, $\Delta_{k}=\E_k-\E_{k-1}$ for $k\geq 1$, and define the square function and maximal function of $f$ respectively by
\beq Sf= \left(\sum_{n=0}^\infty\left|\Delta_n f\right|^2\right)^\frac{1}{2}, \quad f^\ast= \sup_{n}\left|\E_n f\right|.\eeq 
This allows us to define the martingale Hardy space. 
\begin{deff}The space $H^1\left[\left(\mathcal{F}_n\right)_{n=1}^\infty\right]$ is a function space on $\Omega$ with the norm \begin{equation} \left\|f\right\|_{H^1\left[\left(\mathcal{F}_n\right)_{n=1}^\infty\right]}=\E Sf.\end{equation} \end{deff}
We will make use of three following classical martingale inequalities. 
\begin{thm}[Burkholder, Gundy \cite{burgundy} for $1<p<\infty$; Davis \cite{davis} for $p=1$]\label{BDG}For $1\leq p<\infty$,
\beq \left\|Sf\right\|_{L^p}\simeq_p \left\|f^\ast\right\|_{L^p}.\eeq\end{thm}
\begin{thm}[Burkholder \cite{burkh}]\label{burkhfSf}For $1<p<\infty$, 
\beq \left\|f\right\|_{L^p}\simeq_p \left\|Sf\right\|_{L^p}.\eeq\end{thm}
\begin{thm}[Stein \cite{bourgstein}]\label{Stein}For $1<p<\infty$ and an arbitrary sequence $\left(f_n\right)_{n=0}^\infty$, \beq \left\|\sqrt{\sum_{n=0}^\infty\left|\E_n f_n\right|^2}\right\|_{L^p}\lesssim_p \left\|\sqrt{\sum_{n=0}^\infty\left|f_n\right|^2}\right\|_{L^p}.\eeq\end{thm}

\begin{deff}A martingale atom is a function of the form \begin{equation} a=u-\E_{j-1}u,\end{equation} where \begin{equation}\label{eq:atdef}A\in\mathcal{F}_j,\quad \supp u\subset A,\quad \left\|u\right\|_{L^2}\leq |A|^{-\frac{1}{2}}.\end{equation}\end{deff}
\begin{thm}\label{martatdec}Let $f \in H^1\left[\left(\mathcal{F}_n\right)_{n=1}^\infty\right]$ be of mean $0$. Then there are atoms $a_1,a_2,\ldots$ and scalars $c_1,c_2,\ldots$ such that \begin{equation}f=\sum_{n=1}^\infty c_n a_n\end{equation} and \beq \sum_{n=1}^\infty \left|c_n\right|\lesssim \|f\|_{H^1\left[\left(\mathcal{F}_n\right)_{n=1}^\infty\right]}.\eeq \end{thm}

\begin{thm}[Fefferman] \label{feffduality}The dual space to $H^1\left[\left(\mathcal{F}_n\right)_{n=1}^\infty\right]$ is $BMO\left[\left(\mathcal{F}_n\right)_{n=1}^\infty\right]$, where 
\beq \|g\|_{BMO\left[\left(\mathcal{F}_n\right)_{n=1}^\infty\right]}\simeq \sup_k \left\| \sqrt{\E_k \sum_{n\geq k}\left|\Delta_n g\right|^2}\right\|_{L^\infty}, \eeq
where the duality is given by $\langle f,g\rangle=\lim_{n\to \infty}\E\p{\E_n f\E_n g}$. \end{thm}

\textbf{Vector-valued inequalities.} For a Banach space $B$, by $L^p\p{S,B}$ we denote the Bochner space of strongly measurable $B$-valued random variables equipped with the norm 
\beq \|f\|_{L^p(S,B)}=\p{\int_{s}\|f(x)\|_B^p\d\mu(s)}^\frac1p\eeq
(or, equivalently, the closed span of functions of the form $(f\otimes v)(x)= f(x)v$, where $f\in L^p(S)$ and $v\in B$, in the $L^p(S,B)$ norm). For an operator $T$ between subspaces of $L^p\p{S_1}$ and $L^p\p{S_2}$ and a linear operator $F:B_1\to B_2$ we can define $T\otimes F$ and the algebraic tensor product by $\p{T\otimes F}\p{f\otimes v}= T(f)\otimes F(v)$, but this construction does not necessarlily produce a bounded operator on the closure. The main tool for obtaining vector-vlaued extensions of inequalities will be the following lemma, which for $I_1,I_2$ being singletons is due to Marcinkiewicz and Zygmund \cite{marzyg} (in this case $\lesssim \|T\|$ can be replaced with $\leq\|T\|$). 
\begin{lem}\label{rbdd}Let $X_i\subset L^1\p{S_i,\ell^2\p{I_i}}$ for $i=1,2$, $B$ be a Hilbert space and $T:X_1\to X_2$ be bounded. Then $T\otimes \id_B: X_1\otimes B\to X_2\otimes B$, where $X_i\otimes B$ is treated as a subspace of $L^1\p{\Omega_i,\ell^2\p{I_i,B}}$, is bounded with norm $\lesssim \|T\|$. \end{lem}
\begin{proof}Without loss of generality, $B$ is finite-dimensional, say $B=\ell^2\p{J}$ for some finite $J$. Let $X_1\otimes \ell^2\p{J}\ni f=\p{f_j}_{j\in J}$, so that $f_j\in X_1$. Let also $r_j$ for $j\in J$ be Rademacher variables. Then, applying $\ell^2\p{I_2}$-valued Khintchine inequality, 
\begin{align}\left\|\p{T\otimes \id}f\right\|_{L^1\p{S_2,\ell^2\p{I_2\times J}}} =&\int_{S_2}\sqrt{\sum_{j\in J}\left\|Tf_j\p{s}\right\|_{\ell^2\p{I_2}}^2}\d \mu_2(s)\\
\simeq & \int_{S_2}\E\left\|\sum_j r_j Tf_j(s)\right\|_{\ell^2\p{I_2}}\d\mu_2(s)\\
=& \E\int_{S_2}\left\|T\p{\sum_j r_j f_j}(s)\right\|_{\ell^2\p{I_2}}\d\mu_2(s)\\
\leq& \|T\| \E\int_{S_1}\left\|\sum_j r_j f_j(s)\right\|_{\ell^2\p{I_1}}\d\mu_1(s)\\\leq& \|T\|\int_{S_1}\sqrt{\sum_j \left\|f_j(s)\right\|_{\ell^2\p{I_1}}^2}\d\mu_1(s)\\
=&\|T\| \|f\|_{L^1\p{S_1,\ell^2\p{I_1\times J}}}.
\end{align}
\end{proof}

\textbf{Hoeffding decomposition.} Now we define the main object of our interest. In order to avoid technicalities with convergence in strong operator topology, we will work in a finite product of $\Omega$ (all the results extend automatically to $\Omega^\infty$ by density). We will see in a moment that any function $f\in L^1\left(\Omega^n\right)$ can be decomposed in a unique way as
\[ \label{eq:hoeff} f=\sum_{m=0}^n \sum_{1\leq i_1<\ldots<i_m\leq n}P_{i_1,\ldots,i_m}f,\]
where $P_{i_1,\ldots,i_m}f\p{x_1,\ldots,x_n}$ depends only on $x_{i_1},\ldots,x_{i_m}$ and is of mean $0$ with respect to each of $x_{i_1},\ldots,x_{i_m}$ (equivalently, $P_A f$ is $\mathcal{F}_A$-measurable and is orthogonal to all $\mathcal{F}_B$-measurable functions for $B\subsetneq A$). This decomposition has been studied in \cite{bourgwalsh}, \cite{kwap}. In particular, $P_{i_1,\ldots,i_m}$ are pairwise orthogonal orthogonal projections. Let \[P_m= \sum_{1\leq i_1<\ldots<i_m\leq n}P_{i_1,\ldots,i_m}\] and $\V_m$ be the range of $P_m$. It is known \cite{bourgwalsh}, \cite{kwap} that $P_m$ is bounded on $L^p\p{\Omega^n}$, $1<p<\infty$, with norm independent on $n$, but this is not true for $L^1\p{\Omega^n}$. \par
One of the possible ways to prove the existence of the above decomposition in $L^2\left(\Omega^n\right)$ is as follows. First we define the subspace
\beq \V_{\leq m}=\clspan \bigcup_{|A|\leq m} \left\{f\in L^2\left(\Omega^n\right): f\text{ is }\mathcal{F}_A\text{-measurable}\right\}\subset L^2\left(\Omega^n\right)\eeq 
for each $m\geq 0$. The sequence of subspaces $\V_{\leq 0},\V_{\leq 1},\ldots,\V_{\leq n}$ is increasing, so by putting \beq \V_0=\V_{\leq 0}, \quad \V_m=\V_{\leq m}\cap \V_{\leq m-1}^{\perp}\eeq we obtain a decomposition \beq L^2\left(\Omega^n\right)= \bigoplus_{m=0}^{n} \V_m\eeq into an orthogonal direct sum of $\V_m$. We will denote the orthogonal projection onto $\V_m$ by $P_m$. \par
A more explicit formula for $P_m$ can be obtained. For $A\subset [1,n]$, let 
\beq\label{eq:pdef} P_A= \left(\id-\E\right)^{\otimes A}\otimes \E^{\otimes [1,n]\setminus A},\eeq where $\id$ and $\E$ are understood to act on $L^2(\Omega)$, and let $\V_A$ be the range of the projection $P_A$. It is easy to see that 
\beq \label{eq:iddecomp}\E_A= (\id- \E+\E)^{\otimes A}\otimes \E^{\otimes [1,n]\setminus A}=  \sum_{B\subset A}\left(\id-\E\right)^{\otimes B}\otimes\E^{\otimes [1,n]\setminus B}\eeq 
and, since the subspaces $\V_B$ are mutually orthogonal, 
\beq  \label{eq:idoplus}L^2\left(\Omega^n,\mathcal{F}_A\right)= \bigoplus_{B\subset A} \V_B.\eeq
Moreover 
\begin{eqnarray} \V_{\leq m}&=& \clspan \bigcup_{|A|\leq m}L^2\left(\Omega^n,\mathcal{F}_A\right)\\ &=&  \clspan \bigcup_{|A|\leq m}  \bigoplus_{B\subset A} \V_B\\ &=&  \bigoplus_{|B|\leq m} \V_B \end{eqnarray} and consequently \beq\label{eq:pdecomp} \V_m= \bigoplus_{|B|=m} \V_B, \quad P_m=\sum_{|B|=m}P_B.\eeq 

\textbf{Decoupling inequalities.} We are going to present a special case of a theorem of J. Zinn \cite{zinn}, which will be one of the most important tools. \par
\begin{thm}[Zinn]\label{tildemain}For $k=1,\ldots,N$, let $f_k$ be a function on $\Omega^k$. Then
\beq \label{eq:tildesim}\int_{\Omega^N}\sqrt{\sum_{k=1}^N \left|f_k\p{x_1,\ldots,x_k}\right|^2}\d x \simeq \int_{\Omega^N}\int_{\Omega^N}\sqrt{\sum_{k=1}^N \left|f_k\p{x_1,\ldots,x_{k-1},y_k}\right|^2}\d x\d y.\eeq
\end{thm}
We will provide two new proofs of the above in Section 6. Below, we state two corollaries obtained by iterating Zinn's inequality. 
\begin{cor}\label{tildehilb}For $1\leq a< b\leq N$, let $f_{a,b}\in L^1\left(\Omega^N, \mathcal{F}_{[a,b]}\right)$. Denote $\left(x_a,\ldots,x_b\right)$ by $x_{[a,b]}$. Then 
\beq \int_{\Omega^N}\sqrt{\sum_{a<b} \left|f_{a,b}\left(x_{[a,b]}\right)\right|^2}\d x  \simeq \int_{\left(\Omega^N\right)^3}\sqrt{\sum_{a<b} \left|f_{a,b}\left(z_a,x_{[a+1,b-1]},y_b\right)\right|^2}\d x \d y\d z \eeq \end{cor}
\begin{proof}Let $F_b\in L^1\left(\Omega^N, \mathcal{F}_b,\ell^2\right)$ be defined by $\left(F_b\right)_a= f_{a,b}$ for $a<b$ and $0$ otherwise. Then, by Theorem \ref{tildemain} applied for functions $\left\|F_b\right\|_{\ell^2}$, 
\begin{eqnarray} \int_{\Omega^N}\sqrt{\sum_{a< b} \left|f_{a,b}\left(x_{[a,b]}\right)\right|^2}\d x 
&=& \int_{\Omega^N}\sqrt{\sum_{b} \left\|F_{b}\left(x_{\leq b}\right)\right\|^2_{\ell^2}}\d x\\ 
&\simeq&  \int_{\left(\Omega^N\right)^2}\sqrt{\sum_{b} \left\|F_{b}\left(x_{\leq b-1},y_b\right)\right\|^2_{\ell^2}}\d x\d y\\ 
&=& \int_{\left(\Omega^N\right)^2}\sqrt{\sum_{a< b} \left|f_{a,b}\left(x_{[a,b-1]},y_b\right)\right|^2}\d x\d y.\end{eqnarray} 
Analogously, by setting $y$ as fixed, and applying Theorem \ref{tildemain} with reversed order of variables (which we can do, because we are dealing with finite sums),
\beq  \int_{\left(\Omega^N\right)^2}\sqrt{\sum_{a\leq b} \left|f_{a,b}\left(x_{[a,b-1]},y_b\right)\right|^2}\d x\d y\simeq \int_{\left(\Omega^N\right)^3}\sqrt{\sum_{a\leq b} \left|f_{a,b}\left(z_a,x_{[a+1,b-1]},y_b\right)\right|^2}\d x\d y\d z\eeq as desired.\end{proof}
\begin{cor}\label{multizinn}For all $i=\p{i_1,\ldots,i_m}$ such that $i_1<\ldots<i_m$, let $f_i$ be an $\mathcal{F}_{\left[1,i_1-1\right]\cup \left\{i_1,\ldots,i_m\right\}}$-measurable function on $\Omega^\N$. Then, treating each $f_i$ as a function on $\Omega^{\left[1,i_1-1\right]}\times \Omega^{\left\{i_1,\ldots,i_m\right\}}$,
\begin{align}
\label{eq:multidec}&\int_{\Omega^\N}\sqrt{\sum_i \left|f_i\p{x_{<i_1},x_{i_1},\ldots,x_{i_m}}\right|^2}\d x\simeq_m \\
\nonumber &\int_{\Omega^\N}\int_{\p{\Omega^\N}^m}\sqrt{\sum_i \left|f_i\p{x_{<i_1},y^{(1)}_{i_1},\ldots,y^{(m)}_{i_m}}\right|^2}\d y^{(1,\ldots,m)}\d x,
\end{align}
where $y^{(1)},\ldots,y^{(m)}$ are variables in $\Omega^\N$. 
\end{cor}
\begin{proof} Let us fix $k\in\{1,\ldots,m\}$ and for each $j\in \N$ define a function $\varphi_j$ on $\Omega^{\left[1,j\right]}\times \p{\Omega^\N}^{m-k}$ by the formula
\begin{align} &\varphi_j \p{x_{\leq j},y^{(k+1)},\ldots,y^{(m)}}=\\
\nonumber&\sqrt{\sum_{\substack{i_1<\ldots<i_{k-1}<\\ j<i_{k+1}<\ldots<i_{m}}} \left|f_{i_1,\ldots,i_{k-1},j,i_{k+1},\ldots,i_{m}}\p{x_{<i_1},x_{i_1},\ldots,x_{i_{k-1}},x_j,y^{(k+1)}_{i_{k+1}},\ldots,y^{(m)}_{i_m}}\right|^2}.\end{align}
Then, for fixed $y^{(> k)}=\p{ y^{(k+1)},\ldots,y^{(m)}}\in\p{\Omega^\N}^{m-k}$,  
\begin{align}&\int_{\Omega^\N}\sqrt{\sum_{i_1<\ldots<i_m} \left|f_i\p{x_{<i_1},x_{i_1},\ldots,x_{i_{k}},y^{(k+1)}_{i_{k+1}},\ldots,y^{(m)}_{i_m}} \right|^2 }\d x \\
=& \int_{\Omega^\N}\sqrt{\sum_{j\in \N} \left|\varphi_j \p{x_{\leq j},y^{(k+1)},\ldots,y^{(m)}} \right|^2 }\d x \\
\label{eq:appdec2phi}\simeq& \int_{\Omega^\N}\int_{\Omega^\N}\sqrt{\sum_{j\in \N} \left|\varphi_j \p{x_{<j},y^{(k)}_{j},y^{(k+1)},\ldots,y^{(m)}} \right|^2 }\d x\d y^{(k)}\\
=& \int_{\Omega^\N}\int_{\Omega^\N}\sqrt{\sum_{i_1<\ldots<i_m} \left|f_i\p{x_{<i_1},x_{i_1},\ldots,x_{i_{k-1}},y^{(k)}_{i_k},y^{(k+1)}_{i_{k+1}},\ldots,y^{(m)}_{i_m}} \right|^2 }\d x\d y^{(k)}.
\end{align}
Here, $i_k$ plays the role of $j$ and \eqref{eq:appdec2phi} is an application of Theorem \ref{tildemain} to functions $\left|\varphi_j\right|^2$. Integrating the resulting inequality with respect to $y^{(>k)}$, we get 
\begin{align} &\int_{\p{\Omega^\N}^{m-k}}\int_{\Omega^\N}\sqrt{\sum_{i_1<\ldots<i_m} \left|f_i\p{x_{<i_1},x_{i_1},\ldots,x_{i_{k}},y^{(k+1)}_{i_{k+1}},\ldots,y^{(m)}_{i_m}} \right|^2 }\d x\d y^{(\geq k+1)}\\
\nonumber \simeq & \int_{\p{\Omega^\N}^{m-k+1}}\int_{\Omega^\N}\sqrt{\sum_{i_1<\ldots<i_m} \left|f_i\p{x_{<i_1},x_{i_1},\ldots,x_{i_{k-1}},y^{(k)}_{i_k},y^{(k+1)}_{i_{k+1}},\ldots,y^{(m)}_{i_m}} \right|^2 }\d x\d y^{(\geq k)},\end{align}
which by induction from $k=m$ to $k=1$ proves \eqref{eq:multidec}.  \end{proof}

\section{Boundedness of $P_m$ on $L^p\left(\Omega^\N\right)$} \label{secpmlp}

The main motivation for this part is the following theorem, proved by Bourgain with $c_p\lesssim \frac{\hat{p}^{\frac52}}{\ln \hat{p}}$ and by Kwapień with $c_p\lesssim \frac{\hat{p}}{\ln \hat{p}}$, where $\hat{p}=\max\left(p,\frac{p}{p-1}\right)$. 
\begin{thm}[\cite{bourgwalsh}, \cite{kwap}]\label{mainLp}$P_m$ is bounded on $L^p\left(\Omega^\mathbb{N}\right)$ for $1<p<\infty$, with norm $\lesssim c_p^m$. \end{thm}
We will present a proof that yields $\|\acts{P_1}{L^p}\|<\infty$ and $c_p= \e\|\acts{P_1}{L^p}\|$. 

\begin{proof}Without loss of generality, we may assume that we are working in $L^p\left(\Omega^{[1,N]},\mathcal{F}^{\otimes N}\right)$. Indeed, by \eqref{eq:idoplus} and \eqref{eq:pdecomp}, $P_m$ preserves $L^2\left(\Omega^\mathbb{N},\mathcal{F}_{[1,N]}\right)$, which can be canonically identified with $L^2\left(\Omega^{[1,N]},\mathcal{F}^{\otimes N}\right)$. Since the sequence $\left(L^2\left(\Omega^\mathbb{N},\mathcal{F}_{[1,N]}\right):N\in\mathbb{N}\right)$ is increasing and its sum is dense in $L^p\left(\Omega^\mathbb{N},\mathcal{F}^{\otimes\mathbb{N}}\right)$, all we need to prove is \beq\label{eq:induction} \lim_{N\to\infty}\left\|\acts{P_m}{L^p\left(\Omega^{[1,N]},\mathcal{F}^{\otimes N}\right)}\right\|\leq \frac{1}{\e}c_p^m.\eeq \par 

$P_0=\E$ is bounded. The $L^p$ boundedness of $P_1$ is essentially a known result \cite{cajtszrift}, but we provide a proof for the sake of completeness. Let $\left(\mathcal{F}_k:k\in[0,N]\right)$ be the natural filtration and $\left(\mathcal{F}^*_k\right)_{k=0}^N$ be the natural reversed filtration, i.e. $\mathcal{F}^*_k=\mathcal{F}_{[k,N]}$. By \eqref{eq:iddecomp} and \eqref{eq:pdecomp} we see that \beq \label{eq:ugu}P_1=\sum_{k=1}^N P_{\{k\}},\quad \Delta_k=\sum_{\max A=k}P_A,\quad \E^*_k= \sum_{A\subset [k,N]}P_A. \eeq By mutual orthogonality of $P_A$'s \beq \label{eq:goodoldclaim} \Delta_k P_1= P_{\{k\}}= \E_k^*\Delta_k.\eeq  Applying Theorem \ref{burkhfSf}, \eqref{eq:goodoldclaim} and Theorem \ref{Stein}, we obtain \begin{eqnarray} \left\|P_1f\right\|_{L^p} &\simeq_p& \left\|\sqrt{\sum_{k=0}^N \left|\Delta_k P_1 f\right|^2}\right\|_{L^p}\\ &=& \left\|\sqrt{\sum_{k=0}^N \left|\E^*_k \Delta_k f\right|^2}\right\|_{L^p}\\ &\lesssim_p& \left\|\sqrt{\sum_{k=0}^N \left|\Delta_k f\right|^2}\right\|_{L^p} \\ &\simeq_p& \left\|f\right\|_{L^p}.\end{eqnarray}\par
We will now proceed by induction. Suppose that \eqref{eq:induction} is satisfied with $m-1$ in the place of $m$. Let $N=mn$ and define an operator $Q_m$ acting on $L^p\left(\Omega^{[1,N]}\right)$ by 
\beq \label{eq:qdef} Q_{m}=\frac{1}{\binom{N}{n}} \sum_{\substack{A\subset [1,N]\\ |A|=n}} \left(  \left(\acts{P_1}{L^p\left(\Omega^A\right)}\right)\otimes  \left(\acts{P_{m-1}}{L^p\left(\Omega^{[1,N]\setminus A}\right)}\right)\right).\eeq
 Utilising \eqref{eq:pdecomp} we get \begin{eqnarray}   \left(\acts{P_1}{L^p\left(\Omega^A\right)}\right)&=&\sum_{\substack{B_1\subset A\\ \left|B_1\right|=1}} \left(\acts{P_{B_1}}{L^p\left(\Omega^A\right)}\right),\\  \left(\acts{P_{m-1}}{L^p\left(\Omega^{[1,N]\setminus A}\right)}\right)&=&\sum_{\substack{B_2\subset [1,N]\setminus A\\ \left|B_2\right|=m-1}} \left(\acts{P_{B_2}}{L^p\left(\Omega^{[1,N]\setminus A}\right)}\right).\end{eqnarray} 
By \eqref{eq:pdef}, \beq \left(\acts{P_{B_1}}{L^p\left(\Omega^A\right)}\right)\otimes  \left(\acts{P_{B_2}}{L^p\left(\Omega^{[1,N]\setminus A}\right)}\right)=\left( \acts{P_{B_1\cup B_2}}{L^p\left(\Omega^{[1,N]}\right)}\right).\eeq 
Putting the last four equations together, we get 
\begin{eqnarray}Q_{m}&=& \frac{1}{\binom{N}{n}} \sum_{\substack{A\subset [1,N]\\ |A|=n}} \sum_{\substack{B_1\subset A,B_2\subset [1,N]\setminus A\\ \left|B_1\right|=1,\left|B_2\right|=m-1}}\left( \acts{P_{B_1\cup B_2}}{L^p\left(\Omega^{[1,N]}\right)}\right)\\ &=& \frac{1}{\binom{N}{n}} \sum_{\substack{A\subset [1,N]\\ |A|=n}} \sum_{\substack{B\subset [1,N]\\ \left|B\cap A\right|=1,\left|B\setminus A\right|=m-1}}\left( \acts{P_{B}}{L^p\left(\Omega^{[1,N]}\right)}\right) \\ &=& \frac{1}{\binom{N}{n}} \sum_{\substack{B\subset [1,N]\\ \left|B\right|=m}} \left|\left\{A\subset [1,N]: |A|=n, |B\cap A|=1\right\}\right| P_{B}\\&=& \frac{1}{\binom{N}{n}} \sum_{\substack{B\subset [1,N]\\ \left|B\right|=m}}m\binom{N-m}{n-1}P_{B} \\ &=& \label{eq:eins} m\frac{\binom{N-m}{n-1}}{\binom{N}{n}}P_m.\end{eqnarray}
However, by \eqref{eq:qdef} and the induction hypothesis, \begin{eqnarray} \left\|Q_m\right\|&\leq& \frac{1}{\binom{N}{n}} \sum_{\substack{A\subset [1,N]\\ |A|=n}}  \left\|\acts{P_1}{L^p\left(\Omega^A\right)}\right\|\cdot\left\|\acts{P_{m-1}}{L^p\left(\Omega^{[1,N]\setminus A}\right)}\right\|\\ &\leq& \|\acts{P_1}{L^p\left(\Omega^{[1,N]}\right)}\| \frac{c_p^{m-1}}{\e} \\ &=& \label{eq:zwei} \frac{c_p^m}{\e^2}.\end{eqnarray}
Let $a_n\approx b_n$ denote $\lim_{n\to\infty}\frac{a_n}{b_n}=1$. By the Stirling formula, 
\begin{eqnarray} \binom{nm}{n} &\approx& \frac{ \sqrt{2\pi nm} \left(\frac{nm}{\e}\right)^{nm}}{\sqrt{2\pi n}\left(\frac{n}{\e}\right)^n\sqrt{2\pi n(m-1)}\left(\frac{n(m-1)}{\e}\right)^{n(m-1)}}\\ &=& \sqrt{\frac{m}{2\pi n(m-1)}} \frac{n^{nm} m^{nm}}{n^n n^{n(m-1)} (m-1)^{n(m-1)}}\\ &=& \sqrt{\frac{m}{2\pi n(m-1)}} \left(\frac{m^m}{(m-1)^{m-1}}\right)^n.\end{eqnarray}
Thus \begin{eqnarray} \frac{\binom{nm}{n}}{m\binom{(n-1)m}{n-1}} &\approx& \frac{1}{m}\sqrt{\frac{n+1}{n}}\frac{m^m}{(m-1)^{(m-1)}}\\ \label{eq:endstir}&\approx& \left(\frac{m}{m-1}\right)^{m-1}.\end{eqnarray}
Finally, by \eqref{eq:eins}, \eqref{eq:zwei} and \eqref{eq:endstir},
\begin{eqnarray} \lim_{N\to\infty}\left\|\acts{P_m}{L^p\left(\Omega^{[1,N]}\right)}\right\|&=&  \lim_{n\to\infty}\left\|\acts{P_m}{L^p\left(\Omega^{[1,nm]}\right)}\right\|\\ &=&\lim_{n\to\infty} \frac{\binom{nm}{n}}{m\binom{nm-m}{n-1}}\left\|Q_m\right\|\\ &=& \left(1+\frac{1}{m-1}\right)^{m-1} \left\|Q_m\right\| \\ &\leq& \frac{c_p^m}{\e}.\end{eqnarray}
\end{proof}
We prvide a short proof of a fact taken from \cite{cajtszrift} that Theorem \ref{mainLp} can not be, extended to $p=1$ or $\infty$, which motivates the next section.
\begin{prop}\label{nop1onl1}If $\Omega$ is not a single atom, then $P_m$ for $m\geq 1$ is not bounded on $L^1\p{\Omega^\infty}$ or $L^\infty\p{\Omega^\infty}$. \end{prop} 
\begin{proof}It is enough to consider $L^1\p{\Omega^\infty}$, because $P_m$'s are self-adjoint. Let $f\in L^2\p{\Omega}$ be such that $\E f=1$, $f\geq 0$ and $\mu\left(\supp f\right)<1$. Then $\E |f-1|^2>0$. For $F_n=f^{\otimes n}\in L^2\p{\Omega^n}$ we have 
\begin{align}\left\|P_1 F_n\right\|_{L^1\p{\Omega^n}}=& \int_{\Omega^n}\left|\sum_i \p{f\p{x_i}-1}\right|\d x\\ 
\simeq &\int_{\Omega^n}\sqrt{\sum_i \left|f\p{x_i}-1\right|^2}\d x\\
\geq & \sqrt{n \E|f-1|^2}\end{align}
which is not dominated by $\left\|F_n\right\|_{L^1\p{\Omega^n}}= \left\|f\right\|_{L^1\p{\Omega}}^n=1$. To prove the unboundedness of $P_m$ for $m>1$, we simply notice that 
\beq P_m\left((f-1)^{\otimes(m-1)}\otimes F_n\right)= (f-1)^{\otimes (m-1)}\otimes P_1 F_n.\eeq
\end{proof}

\section{Boundedness of $P_m$ on $H^1 \left(\D^\N\right)$}\label{secpmh1}

The projection $P_m$ can be described even more explicitly in the case $\Omega=\T$. Indeed, if $n\in \mathbb{Z}^{\oplus\mathbb{N}}$ is supported on the set $A$, then \beq P_A\e^{i\langle n,t\rangle}= \bigotimes_{j\in A}\left(\id-\E\right)\e^{i n_j t_j}= \prod_{j\in A}\e^{i n_j t_j}= \e^{i\langle n,t\rangle}.\eeq Thus \beq \e^{i \langle n,t\rangle}\in \V_{\supp n}\eeq and \beq \V_m=\clspan \left\{\e^{i \langle n,t\rangle}:|\supp n|=m\right\}.\eeq In particular, $P_m$ preserves the space $H^2\left(\D^\mathbb{N}\right)= L^2\left(\T^\mathbb{N}\right)\cap H^1\left(\D^\mathbb{N}\right)$. \par
In order to adapt the proof of Theorem \ref{mainLp} to the $H^1\left(\D^\N\right)$ case, we will need a replacement for the argument proving that $P_1$ is bounded. The role of the combination of Burkholder-Gundy and Doob inequalities will be played by the following theorem, which can be found in \cite{bourgtelesc}. 
\begin{thm}[Bourgain]\label{bourgsqf}For $f\in H^1_{\mathrm{last}}\left(\T^\mathbb{N}\right)$, there is an equivalence of norms \beq \|f\|_{L^1\left(\T^\mathbb{N}\right)}\simeq \|f\|_{H^1\left[\left(\mathcal{F}_n\right)_{n=0}^\infty\right]},\eeq where $\left(\mathcal{F}_n\right)_{n=0}^\infty$ is the natural filtration on $\T^\mathbb{N}$. \end{thm}
For later use, we note the Hilbert space valued extension.
\begin{cor}\label{bourgsqfl2}Let $B$ be a Hilbert space. For $f\in H^1_{\mathrm{last}}\left(\T^\mathbb{N},B\right)$, there is an equivalence of norms \beq \|f\|_{L^1\left(\T^\mathbb{N},B\right)}\simeq \|f\|_{H^1\left[\left(\mathcal{F}_n\right)_{n=0}^\infty,B\right]}=\int_{\T^\N} \sqrt{\sum_{k=0}^\infty \left\| \Delta_k f(t)\right\|_{B}^2}\d t ,\eeq where $\left(\mathcal{F}_n\right)_{n=0}^\infty$ is the natural filtration on $\T^\mathbb{N}$.\end{cor}
\begin{proof} Theorem \ref{bourgsqf} gives a map 
\beq T: H^{1}_{\mathrm{last}}\p{\T^\N}\to L^1\p{\T^\N,\ell^2},\eeq
which is an isomorphism onto the subspace of $L^1\p{\T^\N,\ell^2}$ consisting of functions $f$ such that $f_k$ is a $k$-th martingale difference and is analytic in the $k$-th variable, defined by 
\beq Tf=\p{\Delta_k f}_{k=0}^\infty.\eeq
Thus, applying Lemma \ref{rbdd} with $I_1$ being a singleton, $I_2=\N$, $T$ as above (and then the same for $T^{-1}$) we get 
\beq \|f\|_{H^1_{\mathrm{last}}\p{\T^\N,B}}\simeq \|\p{T\otimes \id_{B}}f\|_{L^1\p{\T^\N, B}} =\int_{\T^\N} \sqrt{\sum_k \left\| \Delta_k f(t)\right\|_{B}^2}\d t.\eeq
\end{proof}
The role of the Stein martingale inequality will be played by the following simple observation. 
\begin{cor}\label{lepongle}For any sequence $\left(f_n:n\in\mathbb{N}\right)$ adapted to the natural filtration on $\Omega^\mathbb{N}$, \beq \E\sqrt{\sum_{n=1}^\infty\left|f_n\right|^2}\gtrsim \E\sqrt{\sum_{n=1}^\infty\left|\E_{\{n\}}f_n\right|^2}.\eeq\end{cor}
\begin{proof}Let $\tilde{f}_n$ be a sequence of functions on $\Omega^\N\times\Omega^\N$ defined by 
\beq \tilde{f}_n\p{x,y}=f_n\p{x_1,\ldots,x_{n-1},y_n}.\eeq
Applying Theorem \ref{tildemain} and conditional expectation with respect to the second of two sets of variables, 
\begin{eqnarray} \E\sqrt{\sum_{n=1}^\infty\left|f_n\right|^2} &\gtrsim&  \E\sqrt{\sum_{n=1}^\infty \left|\tilde{f}_n\right|^2}\\ &\geq& \E\sqrt{\sum_{n=1}^\infty\left|\left(\E\otimes \id\right)\tilde{f}_n\right|^2}\\ &=& \E\sqrt{\sum_{n=1}^\infty\left|1\otimes \E_{\{n\}} f_n\right|^2}\\ &=& \E\sqrt{\sum_{n=1}^\infty\left|\E_{\{n\}} f_n\right|^2}.\end{eqnarray}\end{proof}
By conditioning with respect to the first set of variables, we obtain the inequality
\beq \E\p{\sum \left|f_n\right|^2}^\frac{1}{2}\gtrsim \E\p{\sum \left|\E_{n-1}f_n\right|^2}^\frac{1}{2}\eeq
due to Lepingle \cite{Lep}. 
\begin{thm}\label{1parameter}For any $\Omega$, $P_1$ is bounded on $H^1\left[\left(\mathcal{F}_n\right)_{n=0}^\infty\right]$. \end{thm}
\begin{proof}We proceed as in the proof of Theorem \ref{mainLp}. First, we reduce the problem to the $\Omega^{[1,N]}$ realm. Then we notice that \beq \Delta_k P_1= P_{\{k\}}= \E_{\{k\}}\Delta_k,\eeq which by Corollary \ref{lepongle} yields \begin{eqnarray} \left\|P_1f\right\|_{H^1} &=& \left\|\sqrt{\sum_{k=0}^N \left|\Delta_k P_1 f\right|^2}\right\|_{L^1}\\ &=& \left\|\sqrt{\sum_{k=0}^N \left|\E_{\{k\}} \Delta_k f\right|^2}\right\|_{L^1}\\ &\lesssim& \left\|\sqrt{\sum_{k=0}^N \left|\Delta_k f\right|^2}\right\|_{L^1} \\ &=& \left\|f\right\|_{H^1}.\end{eqnarray}\end{proof}

\begin{thm}\label{PmonH1Dinfty}$P_m$ is bounded on $H^1\left(\mathbb{D}^\mathbb{N}\right)$ with norm $\leq \frac{1}{\e}c_1^m$, where \beq c_1=\e \left\|\acts{P_1}{H^1\left(\D^\mathbb{N}\right)}\right\|.\eeq \end{thm} 
\begin{proof}The case $m=0$ is trivial, $m=1$ follows directly from Theorem \ref{bourgsqf} and Theorem \ref{1parameter}. The induction step is identical to the proof of Theorem \ref{mainLp}, up to changing $L^p\left(\Omega^I\right)$ to $H^1\left(\D^I\right)$. Alternatively, we can prove the same in a single step. 
Set 
\beq Q_{m,n}=\frac{1}{\binom{nm}{n,\ldots,n}}\sum_{\substack{A_1\cup\ldots\cup A_m=[1,nm] \\ A_i \text{'s disjoint} \\ \left|A_1\right|=\ldots= \left|A_m\right|=n }}\bigotimes_{i=1}^m \p{\acts{P_1}{H^1\p{\D^{A_i}}}}.\eeq 
It is easily seen that for each set $B$ of cardinality $m$, $P_B$ appears $m!\binom{(n-1)m}{n-1,\ldots,n-1}$ times in the sum. Therefore
\begin{eqnarray} \left\|P_1:H^1\p{\D^\infty}\right\|^m&\geq& \left\|Q_{m,n}\right\| \\
&=& \frac{m!\binom{(n-1)m}{n-1,\ldots,n-1}}{\binom{nm}{n,\ldots,n}}\left\|\acts{P_m}{H^1\p{\D^{nm}}}\right\|\end{eqnarray}
and since 
\begin{eqnarray} \lim_{n\to \infty} \frac{\binom{nm}{n,\ldots,n}}{m!\binom{(n-1)m}{n-1,\ldots,n-1}} &=& \lim_{n\to \infty} \frac{\frac{(nm)!}{n!^m}}{m!\frac{(nm-m)!}{(n-1)!^m}}\\
&=& \lim_{n\to \infty} \frac{nm}{m}n^{-m}\\
&=& \frac{m^m}{m!} \end{eqnarray}
we get 
\beq \left\|\acts{P_m}{H^1\p{\D^\infty}}\right\|\leq \frac{m^m}{m!}\left\|P_1:H^1\p{\D^\infty}\right\|^m\leq \frac{1}{\sqrt{2\pi m}}\p{\e \left\|P_1:H^1\p{\D^\infty}\right\|}^m.\eeq
  \end{proof}
It has to be noted that our proofs of Theorems \ref{mainLp} and \ref{PmonH1Dinfty} extend naturally to a vector valued case, respectively UMD and AUMD valued. Indeed, Bourgain's proof of Theorem \ref{Stein}, as presented in \cite{mullbook}, extends to the UMD valued version, while Theorem \ref{bourgsqf} is just the statement that a one-dimensional space has the AUMD property. In both cases, the induction follows without change. There is also a second direction in which we can generalize. Namely, by looking carefully at the proof of Theorem \ref{bourgsqf}, one can see that the only place in which analyticity plays a role is the $H^1=H^2\cdot H^2$ theorem, which is true for $H^1$ on any compact and connected group with ordered dual \cite{rudingroup}, which means that we can replace $\T$ with any such group. \par
Given that Kwapień's constant $c_p$ in Theorem \ref{mainLp} has the best known asymptotics as a function of $p$ for $m=1$, one can ask about the dependence of $\left\|\acts{P_m}{L^p\p{\Omega^\infty}}\right\|$ and $\left\|\acts{P_m}{H^1\p{\D^\N}}\right\|$ on $m$.
\begin{prop}The inequalities
\beq \label{eq:pmlplowerbd}\left\|\acts{P_m}{L^p\p{\Omega^\infty}}\right\|\geq \left\|\acts{P_1}{L^p_0\p{\Omega^\infty}}\right\|^m\eeq
for nontrivial $\Omega$ and 
\beq \label{eq:pmh1lowerbd}\left\|\acts{P_m}{H^1\p{\D^\N}}\right\|\geq \left\|\acts{P_1}{H^1_0\p{\D^\N}}\right\|^m,\eeq
where $L^p_0$ and $H^1_0$ stand for functions of mean $0$, are true. Also, 
\beq \label{eq:p1h1notcontr}\left\|\acts{P_1}{H^1_0\p{\D^\N}}\right\|>1.\eeq \end{prop}
\begin{proof}Let $f\in L^p\p{\Omega^n}$ be of mean $0$. Then $f^{\otimes m}\in L^p\p{\Omega^{mn}}$ and
\beq\label{eq:pmfmp1fm} \p{\acts{P_m}{L^p\p{\Omega^{mn}}}}\p{f^{\otimes m}}= \p{ \p{\acts{P_1}{L^p\p{\Omega^n}}}\p{f}}^{\otimes m}.\eeq 
Indeed, we have $f=\sum_{|A|\geq 1}P_A f$ because of $\E f=0$, hence 
\beq f^{\otimes m}= \sum_{\substack{A_i\subset [n(i-1)+1,ni]\\ \left|A_i\right|\geq 1\text{ for }i=1,\ldots,m}} \bigotimes_{i=1}^m P_{A_i}f\eeq 
The only way to get a summand in $\V_m$ is to have $\left|A_i\right|=1$ for all $i$ and the sum of such summands is the right hand side of \eqref{eq:pmfmp1fm}. Taking an $f$ which is close to attaining the norm of $P_1$ on a respective space proves \eqref{eq:pmlplowerbd} and \eqref{eq:pmh1lowerbd}. \par
In order to see \eqref{eq:p1h1notcontr}, assume for the sake of contradiction that $P_1$ is a contraction on $H^1_0\p{\D^2}$. We will test it on functions of the form $F(z)+w+azw$, where $F\in H^1_0\p{\D}$ and $a$ is a scalar. It is easy to see that 
\beq \E\left|\alpha+ \beta w\right|=\E\left||\alpha|+|\beta|w\right|\eeq
for $\alpha, \beta\in\mathbb{C}$. Hence, from the inequality
\beq \label{eq:ctrexageq0}\E\left|F(z)+w(1+az)\right|\geq \E\left|F(z)+w\right|\eeq
we get
\beq \E\left|\left|F(z)\right|+w(1+az)\right|\geq \E\left|\left|F(z)\right|+w\right|.\eeq
Since any nonnegative function can be approximated by the modulus of an $H^1_0\p{\D}$ function, \eqref{eq:ctrexageq0} is true for any nonnegative $F$. In particular, the left hand side attains a local minimum at $a=0$, so by $|u+v|=|u|+\Re \frac{u\overline{v}}{|u|}+o(v)$ we infer that
\beq \label{eq:ctrexlagr}\Re\E \frac{\p{F(z)+w}\overline{wz}}{|F(z)+w|}=0.\eeq
Now let 
\beq \phi(r)= \E \frac{1+r\overline{w}}{|1+r\overline{w}|}\eeq
for $r\geq 0$. This is a continuous function, whose values lie on some curve $\gamma$ connecting $0$ and $1$ (because $\phi(0)=1$ and $\lim_{r\to\infty}\phi(r)=0$). The condition \eqref{eq:ctrexlagr} can be rewritten as 
\beq \label{eq:ctrexBS}\Re\E \overline{z}\phi(F(z))=0.\eeq
Since $F$ was allowed to be any positive function, $\phi(F)$ can be any function with values in $\gamma$, making \eqref{eq:ctrexBS} obviously false. \end{proof}
\section{Martingale Hardy spaces}
\subsection{Double indexed martingales}\label{ssdoubleindex}Above we noticed that the boundedness of $P_1$ on $H^1\left(\D^\mathbb{N}\right)$ follows from the boundedness of $P_1$ on a bigger space $H^1\left[\left(\mathcal{F}_n\right)\right]$. It is temtping to find an abstract martingale inequality responsible for the boundedness of $P_m$ on $H^1\left(\D^\mathbb{N}\right)$. We can do this for $m=2$.\par
By the natural double-indexed filtration on $\Omega^\mathbb{N}$ we will mean the family $\left(\mathcal{F}_{[a,b]}:a\leq b\right)$ (note that the inclusion order in the first index is reversed). Let $\Delta_n=\E_n-\E_{n-1}$ be the martinagle differences with respect to $\left(\mathcal{F}_n\right)$ and $\Delta^*_n=\E^*_n-\E^*_{n+1}$ be the martingale differences with repsect to $\left(\mathcal{F}^*_n\right)$, where $\mathcal{F}^*_n=\mathcal{F}_{[n,\infty)}$. We define the martingale differences with respect to $\left(\mathcal{F}_{[a,b]}\right)$ by \beq \Delta_{[a,b]}= \Delta_{a}^*\Delta_b = \E_{[a+1,b-1]}+\E_{[a,b]}- \E_{[a+1,b]}-\E_{[a,b-1]}\eeq and an $H^1$ norm for this filtration by \beq \|f\|_{H^1\left[\left(\mathcal{F}_{[a,b]}\right)\right]}= \E\sqrt{\left|\E f\right|^2+\sum_{1\leq a\leq b} \left|\Delta_{[a,b]}f\right|^2}.\eeq 
The definition of double martingale differences coincides with what is considered in \cite{niestety}. 
\begin{cor}\label{rzsqf}For $f\in H^1\left(\D^\mathbb{N}\right)$, there is an equivalence of norms \beq \|f\|_{H^1\left(\D^\mathbb{N}\right)}\simeq \|f\|_{H^1\left[\left(\mathcal{F}_{a,b}\right)_{a\leq b}\right]},\eeq where $\left(\mathcal{F}_{[a,b]}\right)_{a\leq b}$ is the natural double-indexed filtration on $\T^\mathbb{N}$. \end{cor}
\begin{proof}For any $\pm 1$-valued sequence $\left(\varepsilon_n:n\in\mathbb{N}\right)$, we define operators $S_\varepsilon$ and $S^*_\varepsilon$ by \beq S_\varepsilon f=\E f+ \sum_{n=1}^\infty \varepsilon_n \Delta_n f, \quad S^*_\varepsilon f=\E f+ \sum_{n=1}^\infty \varepsilon_n \Delta^*_n f.\eeq By Theorem \ref{bourgsqf}, $S_\varepsilon$ is an isomorphism from $H^1\left(\mathbb{D}^\mathbb{N}\right)$ to itself, uniformly in $\varepsilon$. By reversing the order of variables, the same can be said about $S^*_\varepsilon$. Thus for any $\varepsilon, \varepsilon'$, \begin{eqnarray} \|f\|_{H^1\left(\D^\mathbb{N}\right)}&\simeq& \left\|S_\varepsilon S^*_{\varepsilon'}f\right\|_{H^1\left(\D^\mathbb{N}\right)}\\ &=&\left\| \E f+\sum_{a\leq b}\varepsilon_a\varepsilon'_b \Delta^*_a\Delta_b f\right\|_{H^1\left(\D^\mathbb{N}\right)}\\ &=& \E\left|\E f+\sum_{a\leq b}\varepsilon_a\varepsilon'_b \Delta_{[a,b]}f\right|. \end{eqnarray} By averaging the last quantity over all choices of $\varepsilon, \varepsilon'$ and applying the Khintchine-Kahane inequality twice, we get the desired inequalities.
\end{proof}
\begin{thm}$P_2$ is bounded on $H^1\left[\left(\mathcal{F}_{a,b}\right)_{a\leq b}\right]$, for any $\Omega$.\end{thm}
\begin{proof}As usual, we reduce the problem to the $\Omega^{[1,N]}$ version. By \eqref{eq:ugu}, \beq \Delta_{[a,b]}=\sum_{\substack{\min A=a\\ \max A=b}}P_{A}.\eeq Thus \beq \Delta_{[a,b]}P_2= P_{\{a,b\}}= \E_{\{a,b\}}\Delta_{[a,b]}\eeq for $a<b$.  We can assume that $P_0 f=P_1 f=0$ (i.e. $\E f=0$ and $\Delta_{[a,a]}f=0$ for all $a$), because $\V_{\leq 1}$, being the image of $\E+ \sum_a \Delta_{[a,a]}$ is trivially complemented in the underlying norm and $P_2$ is $0$ on $\V_{\leq 1}$. By applying Corollary \ref{tildehilb}, 
\begin{eqnarray} \|f\|_{H^1}&=& \E\sqrt{\sum_{a< b}\left|\Delta_{[a,b]}f\right|^2}\\ &=& \int_{\Omega^N} \sqrt{\sum_{a<b} \left|\Delta_{[a,b]}f\left(x_{[a,b]}\right)\right|^2}\d x\\ 
&\simeq& \int_{\left(\Omega^N\right)^3} \sqrt{\sum_{a<b} \left|\Delta_{[a,b]}f\left(z_a,x_{[a+1,b-1]},y_b\right)\right|^2}\d x \d y \d z\\ 
&\geq& \int_{\left(\Omega^N\right)^2} \sqrt{\sum_{a<b} \left|\int_{\Omega^N}\Delta_{[a,b]}f\left(z_a,x_{[a+1,b-1]},y_b\right)\d x\right|^2}\d y \d z\\ 
&\simeq & \E \sqrt{\sum_{a<b}\left|\E_{\{a,b\}}\Delta_{[a,b]}f\right|^2}\\ 
&=& \E \sqrt{\sum_{a<b}\left|\Delta_{[a,b]}P_2 f\right|^2}\\ 
&=& \left\|P_2f\right\|_{H^1}
\end{eqnarray} as desired.\end{proof} 
\subsection{Multiple indexed martingales}
We will make an attempt at generalizing the above for multiple indexed martingales. Suppose there is a family $\p{T_i,\partial T_i}_{i\in \mathcal{I}}$ of pairs of finite subsets of some set $X$ (finite or not) indexed by some set $\mathcal{I}$, such that $\partial T_i\subseteq T_i$ ($\partial T_i$ is not a boundary in a topological sense - we use this notation for resemblance with the case where $T_i$ are intervals and $\partial T_i$ are their endpoints). We would like to define operators $\Delta_i$ on $L^2\p{\Omega^X}$ by the formula
\beq \label{eq:deltaigendef}\Delta_i= \p{\id-\E}^{\otimes \partial T_i}\otimes \id^{\otimes T_i\setminus \partial T_i}\otimes \E^{\otimes T_i'},\eeq 
where $T_i'$ stands for the complement of $T_i$ in $X$. This is supposed to mimic the standard martingale differences when $X=\N$, $\mathcal{I}=\N$, $T_i=[0,i]$, $\partial T_i=\{i\}$ and double martingale differences when $\mathcal{I}=\left\{\p{a,b}:a\leq b\right\}$, $T_{a,b}=[a,b]$, $\partial T_{a,b}=\{a,b\}$. The natural condition
\beq \sum_i \Delta_{i} f= f\eeq 
is guaranteed by 
\beq \label{eq:Tcond}\text{for any }A\subset X, \text{there exists unique }i\in \mathcal{I}\text{ such that } \partial T_i\subseteq A\subseteq T_i.\eeq
Indeed, 
\begin{eqnarray} \Delta_i&=& \p{\id-\E}^{\otimes \partial T_i}\otimes \id^{\otimes T_i\setminus \partial T_i}\otimes \E^{\otimes T_i'}\\ 
&=&\p{\id-\E}^{\otimes \partial T_i} \otimes \E^{\otimes T_i'}\otimes\sum_{B\subseteq T_i\setminus \partial T_i} \p{\acts{P_B}{L^2\p{\Omega^{T_i\setminus \partial T_i}}}}\\
&=&\p{\id-\E}^{\otimes \partial T_i} \otimes \E^{\otimes T_i'}\otimes\sum_{B\subseteq T_i\setminus \partial T_i} \p{\id-\E}^{\otimes B}\otimes \E^{\otimes T_i\setminus\p{\partial T_i\cup B}}\\
&=& \label{eq:deltaintop}\sum_{B\subseteq T_i\setminus \partial T_i}P_{\partial T_i\cup B}.\end{eqnarray}
Hence
\beq \sum_{i\in \mathcal{I}} \Delta_i= \sum_{i\in \mathcal{I}} \sum_{B\subseteq T_i\setminus \partial T_i}P_{\partial T_i\cup B}\eeq
and each $P_A$ appears in the above sum exactly once if and only if the condition \eqref{eq:Tcond} is satisfied. 
For a family $\p{T_i,\partial T_i}_{i\in\mathcal{I}}$ we may define a norm by the formula
\beq \|f\|_{H^1\left[\p{T_i,\partial T_i}_{i\in\mathcal{I}}\right]}= \E\sqrt{\sum_{i\in\mathcal{I}} \left|\Delta_i f\right|^2}\eeq and ask the following:
\begin{itemize}
\item Is it true that \beq \label{eq:DeltaTequivH1}\|f\|_{H^1\left[\p{T_i,\partial T_i}_{i\in\mathcal{I}}\right]}\simeq \|f\|_{H^1\p{\D^\N}} \eeq for $f\in \|f\|_{H^1\p{\D^\N}}$?
\item If yes, is there any interesting example of a set $\mathbb{N}^{\oplus \infty}\subset \Gamma\subset \mathbb{Z}^{\oplus \infty}$ such that \eqref{eq:DeltaTequivH1} is true for $f\in L^1\p{\T^\infty}$ with $\supp \widehat{f}\subset \Gamma$?
\item For which, if any, $m$ is $P_m$ bounded on $H^1\left[\p{T_i,\partial T_i}_{i\in\mathcal{I}}\right]$?
\end{itemize}
We are able to answer them in the case when 
\beq \label{eq:mlastidef}\mathcal{I}=\left\{A\subset \N: |A|\leq m\right\}\eeq
\beq \label{eq:mlasttdef} \partial T_A=A,\quad T_A= \left\{ \begin{array}{lcr}A&\text{ if }&|A|<m,\\ \left(0,\min A\right)\cup A&\text{ if }& |A|=m.\end{array}\right.\eeq
For a finite set $B\subset \N$, the unique $A\in\mathcal{I}$ such that $\partial T_A\subseteq B\subseteq T_A$, which we will denote by $\partial B$, is
\beq \partial B=\left\{ \begin{array}{lcr} B&\text{ if }&|B|<m,\\ m\text{ last elements of }B& \text{ if }& |B|\geq m.\end{array}\right.\eeq
\begin{thm}\label{mainthmh1mlast}Let $m\geq 1$ be fixed and $\p{T_i,\partial T_i}_{i\in\mathcal{I}}$ be defined by \eqref{eq:mlastidef}, \eqref{eq:mlasttdef}. Then
\beq\label{eq:sqfunmlast}\|f\|_{H^1_{m\,\mathrm{last}}\p{\T^\N}}\simeq \|f\|_{H^1\left[\p{T_i,\partial T_i}_{i\in\mathcal{I}}\right]}\eeq
for $f\in H^1_{m\,\mathrm{last}}\p{\T^\N}$, where $\T$ is used as $\Omega$. Moreover, for $m'\in \N$ and nontrivial $\Omega$, the following are equivalent. \\
(i) $m'\leq m$ \\
(ii) $P_{m'}$ is bounded on $H^1\left[\p{T_i,\partial T_i}_{i\in\mathcal{I}}\right]$\\
(iii) $P_{m'}$ is bounded on $H^1_{m\,\mathrm{last}}\p{\T^\N}$. \end{thm}
\begin{proof} For $A,B\subset \N$, we will write $T^{(m)}_A$, $\partial T^{(m)}_A$, $\Delta^{(m)}_A$, $\partial^{(m)}B$ to indicate the value of $m$ we are currently using. For brevity we will denote $\p{T^{(m)}_A,\partial T^{(m)}_A}_{|A|\leq m}$ by $\mathcal{T}_m$. For $|A|<m$, we have $\Delta^{(m)}_A=\p{\id-\E}^{\otimes A}\otimes \E^{\N\setminus A}= P_A$. In particular, $\Delta^{(m)}_{\emptyset}= \E$. Therefore, by definition of the $H^1\left[\mathcal{T}_m\right]$ norm and Corollary \ref{multizinn}, 
\begin{align}\label{eq:h1multidec} \|f\|_{H^1\left[\mathcal{T}_m\right]}\simeq &\int_{\Omega^\N}\int_{\p{\Omega^\N}^m} \sqrt{\sum_{i_1<\ldots<i_m}\left|\Delta^{(m)}_i f\p{x_{<i_1},y^{(1)}_{i_1},\ldots,y^{(m)}_{i_m}}\right|^2}\d y^{(1,\ldots,m)}\d x\\
\nonumber &+\sum_{0<s<m} \int_{\p{\Omega^\N}^s}\sqrt{\sum_{i_1<\ldots<i_s}\left|P_i f\p{y^{(1)}_{i_1},\ldots,y^{(s)}_{i_s}}\right|^2}\d y^{(1,\ldots,s)}+ \left|\E f\right|. \end{align}
Here, we identify an increasing sequence with the set of its elements, write $\d y^{(1,\ldots,m)}$ to denote $\d y^{(1)}\ldots \d y^{(m)}$ and treat $\Delta^{(m)}_if$ as a function on $\Omega^{\left[1,i_1-1\right]}\times \Omega^{\left\{i_1,\ldots,i_m\right\}}$. From this expression, we immediately see the implication $(i)\implies (ii)$. Indeed, for $m'<m$, 
\beq \Delta^{(m)}_A P_{m'}=\left\{\begin{matrix}\Delta^{(m)}_A&\text{ if }&|A|=m'\\ 0&\text{ if }&|A|\neq m',\end{matrix}\right.\eeq 
which trivializes the inequality $\|f\|_{H^1\left[\mathcal{T}_m\right]}\gtrsim \|P_{m'}f\|_{H^1\left[\mathcal{T}_m\right]}$. For $m'=m$, we notice that 
\beq \Delta^{(m)}_A P_{m}=\left\{\begin{matrix}\E^{\otimes \left[1,\min A-1\right]}\Delta^{(m)}_A&\text{ if }&|A|=m\\ 0&\text{ if }&|A|\neq m,\end{matrix}\right.\eeq 
and the desired inequality follows from
\begin{align} \|f\|_{H^1\left[\mathcal{T}_m\right]}\geq & \int_{\Omega^\N}\int_{\p{\Omega^\N}^m} \sqrt{\sum_{i_1<\ldots<i_m}\left|\Delta^{(m)}_i f\p{x_{<i_1},y^{(1)}_{i_1},\ldots,y^{(m)}_{i_m}}\right|^2}\d y^{(1,\ldots,m)}\d x\\
\geq & \int_{\p{\Omega^\N}^m} \sqrt{\sum_{i_1<\ldots<i_m}\left|\int_{\Omega^\N}\Delta^{(m)}_i f\p{x_{<i_1},y^{(1)}_{i_1},\ldots,y^{(m)}_{i_m}}\d x\right|^2}\d y^{(1,\ldots,m)}\\
=& \int_{\p{\Omega^\N}^m} \sqrt{\sum_{i_1<\ldots<i_m}\left|\Delta^{(m)}_i P_m f\p{x_{<i_1},y^{(1)}_{i_1},\ldots,y^{(m)}_{i_m}}\right|^2}\d y^{(1,\ldots,m)}\\
=& \|P_m f\|_{H^1\left[\mathcal{T}_m\right]}.
\end{align}
\par
The implication $(ii)\implies (iii)$ follows from \eqref{eq:sqfunmlast}, which we will prove by induction with respect to $m$. For $m=1$ this is just Theorem \ref{bourgsqf}. Suppose it is true for some $m$ and let $f\in H^1_{m+1\,\mathrm{last}}\p{\T^\N}$. In particular, $f\in H^1_{m\,\mathrm{last}}\p{\T^\N}$. By \eqref{eq:sqfunmlast}, which is now the induction hypothesis, and \eqref{eq:h1multidec}, 
\begin{align}\label{eq:indhyph1ml} \|f\|_{L^1\p{\T^\N}}\simeq &\int_{\T^\N}\int_{\p{\T^\N}^m} \sqrt{\sum_{i_1<\ldots<i_m}\left|\Delta^{(m)}_i f\p{x_{<i_1},y^{(1)}_{i_1},\ldots,y^{(m)}_{i_m}}\right|^2}\d y^{(1,\ldots,m)}\d x\\
\nonumber &+\sum_{0<s<m} \int_{\p{\T^\N}^s}\sqrt{\sum_{i_1<\ldots<i_s}\left|P_i f\p{y^{(1)}_{i_1},\ldots,y^{(s)}_{i_s}}\right|^2}\d y^{(1,\ldots,s)}+ \left|\E f\right|. \end{align}
The last two summands are as they are in the desired expression for $m+1$ instead of $m$ and we only have to deal with the first. For any $i_1<\ldots<i_m$ and $t\in\T^\N$ we have 
\beq \Delta^{(m)}_{i_1,\ldots,i_m}f(t)= \sum_{\partial^{(m)}\supp n=i }\widehat{f}(n)\e^{i\langle n,t\rangle}.\eeq
Thus, treating $\Delta^{(m)}_i f$ as a function on $\T^{\left[1,i_1-1\right]}\times \T^{\left\{i_1,\ldots,i_m\right\}}$,
\beq \label{eq:deltamxy}\Delta^{(m)}_{i_1,\ldots,i_m}f\p{x_{<i_1},y_1,\ldots,y_m}= \sum_{\partial^{(m)}\supp n=i }\widehat{f}(n)\e^{i\sum_{j<i_1}n_j x_j+i\sum_{1\leq j\leq m}n_{i_j}y_j }.\eeq
Let $y$ be fixed and $\Delta^{(1)}_k$, where $k\in \N\cup\{\emptyset\}$, act with respect to the variable $x\in\T^\N$ (so, technically, $\Delta^{(1)}_k$ stands for $\Delta^{(1)}_k\otimes \id$). Then 
\beq \label{eq:d1dm1}\Delta^{(1)}_\emptyset \Delta^{(m)}_{i_1,\ldots,i_m}f\p{x_{<i_1},y_1,\ldots,y_m} = P_i f\p{y_1,\ldots,y_m}\eeq
and 
\beq \label{eq:d1dm2}\Delta^{(1)}_k \Delta^{(m)}_{i_1,\ldots,i_m}f\p{x_{<i_1},y_1,\ldots,y_m}=0\text{ for }k\geq i_1.\eeq
For $k<i_1$,
\begin{align} \lefteqn{\Delta^{(1)}_k\Delta^{(m)}_{i_1,\ldots,i_m}f\p{x_{<i_1},y_1,\ldots,y_m} }\\
=&\sum_{\partial^{(m)}\supp n=i }\widehat{f}(n)\e^{i\sum_{1\leq j\leq m}n_{i_j}y_j }\Delta^{(1)}_k\e^{i\sum_{j<i_1}n_j \cdot_j}\p{x}\\
=& \sum_{\substack{\partial^{(m)}\supp n=i\\ \max \p{\supp n\setminus i}=k }}\widehat{f}(n)\e^{i\sum_{1\leq j\leq m}n_{i_j}y_j +i\sum_{j\leq k}n_j x_j}\\
=& \sum_{\partial^{(m+1)}\supp n=\left\{k,i_1,\ldots,i_m\right\}}\widehat{f}(n) \e^{i\sum_{j<k}n_j x_j + in_k x_k +i \sum_{1\leq j\leq m}n_{i_j}y_j}\\
=& \label{eq:d1dm3}\Delta^{(m+1)}_{k,i_1,\ldots,i_m}f\p{x_{<k},x_k,y_1,\ldots,y_m}.
\end{align}
By \eqref{eq:deltamxy}, $\Delta^{(m)}_{i_1,\ldots,i_m}f\p{x_{<i_1},y_1,\ldots,y_m}$ is in $H^1_{\mathrm{last}}\p{\T^\N}$ with respect to $x$. Therefore, applying Corollary \ref{bourgsqfl2} to the vector valued function $x\mapsto\p{\Delta^{(m)}_{i}f\p{x_{<i_1},y^{(1)}_{i_1},\ldots,y^{(m)}_{i_m}}}_{i_1<\ldots<i_m}$ with fixed $y^{(1,\ldots,m)}$, plugging in \eqref{eq:d1dm1}, \eqref{eq:d1dm2}, \eqref{eq:d1dm3} and using Corollary \ref{multizinn}, we get 
\begin{align} \lefteqn{ \int_{\T^\N} \sqrt{\sum_{i_1<\ldots<i_m}\left|\Delta^{(m)}_i f\p{x_{<i_1},y^{(1)}_{i_1},\ldots,y^{(m)}_{i_m}}\right|^2}\d x}\\
=& \int_{\T^\N} \left\|\p{ \Delta^{(m)}_i f\p{x_{<i_1},y^{(1)}_{i_1},\ldots,y^{(m)}_{i_m}}}_{i_1<\ldots<i_m}\right\|_{\ell^2}\d x\\
\simeq&  \left\|\p{ \Delta^{(1)}_\emptyset \Delta^{(m)}_i f\p{x_{<i_1},y^{(1)}_{i_1},\ldots,y^{(m)}_{i_m}}}_{i_1<\ldots<i_m}\right\|_{\ell^2}\\
&+ \int_{\T^\N}\sqrt{\sum_k \left\|\p{ \Delta^{(1)}_k\Delta^{(m)}_i f\p{x_{<i_1},y^{(1)}_{i_1},\ldots,y^{(m)}_{i_m}}}_{i_1<\ldots<i_m}\right\|_{\ell^2}^2}\d x\\
=&\left\|\p{ P_{i_1,\ldots,i_m} f\p{y^{(1)}_{i_1},\ldots,y^{(m)}_{i_m}}}_{i_1<\ldots<i_m}\right\|_{\ell^2}\\
&+ \int_{\T^\N} \sqrt{\sum_{k<i_1<\ldots<i_m}\left|\Delta^{(m+1)}_{k,i_1,\ldots,i_m} f\p{x_{<k}, x_{k},y^{(1)}_{i_1},\ldots,y^{(m)}_{i_m}}\right|^2}\d x\\
=& \sqrt{\sum_{i_1<\ldots<i_m}\left|P_{i_1,\ldots,i_m}f \p{y^{(1)}_{i_1},\ldots,y^{(m)}_{i_m}}\right|^2} \\
+& \int_{\T^\N} \int_{\T^\N}\sqrt{\sum_{k<i_1<\ldots<i_m}\left|\Delta^{(m+1)}_{k,i_1,\ldots,i_m} f\p{x_{<k}, y^{(0)}_{k},y^{(1)}_{i_1},\ldots,y^{(m)}_{i_m}}\right|^2}\d x\d y^{(0)}.
\end{align}
Integrating the resulting equivalence with respect to $y^{(1,\ldots,m)}$ and plugging into \eqref{eq:indhyph1ml}, we verify that $\|f\|_{L^1\p{\T^\N}}\simeq H^1\left[\mathcal{T}_{m+1}\right]$, which finishes the proof of \eqref{eq:sqfunmlast}. \par
In order to see that $(iii)\implies (i)$, let us take $m'>m$. For any $g\in L^1\p{\T^n}$, the function $G\in L^1\p{\T^\N}$ defined by 
\beq G\p{t}=g\p{t_1,\ldots,t_n}e^{i\sum_{j=n+1}^{n+m}t_j}\eeq
is in $H^1_{m\,\mathrm{last}}\p{\T^\N}$. But 
\beq \p{P_{m'}G}(t)=\p{P_{m'-m}g}\p{t_1,\ldots,t_n}e^{i\sum_{j=n+1}^{n+m}t_j},\eeq
so 
\begin{align} \left\|\acts{P_{m'}}{H^1_{m\,\mathrm{last}}\p{\T^\N}}\right\|\geq &\frac{\int_{\T^\N} \left|\p{P_{m'}G}(t)\right|\d t}{\int_{\T^\N} \left|G(t)\right|\d t}\\ 
=& \frac{\int_{\T^\N}\left|\p{P_{m'-m}g}\p{t_1,\ldots,t_n}\right|\d t}{\int_{\T^\N}\left|g\p{t_1,\ldots,t_n}\right|\d t},\end{align}
which by Proposition \ref{nop1onl1} can be arbitrarily big.\end{proof}
It is worth noting that by repeating the above proof of the equivalence between $H^1_{m\text{ last}}$ norm and $H^1\left[\mathcal{T}_m\right]$, one can obtain
\beq \left\|f\right\|_{H^p\left[\mathcal{T}_m\right]}\simeq_p \|f\|_{L^p}\eeq
where $H^p\left[\mathcal{T}_m\right]$ is defined in a natural way. Moreover, by iterating the $\|f\|_{H^1}\geq \|f\|_{L^1}$ inequality for linearly ordered martingales,
\beq \|f\|_{H^1\left[\mathcal{T}_m\right]}\gtrsim \|f\|_{L^1}.\eeq
\section{Appendix}\label{wyrostek}
We present two proofs of Theorem \ref{tildemain} different from the original one by Zinn. \par
Let us recall the non-linear telescoping lemma due to Bourgain and Müller. 

\begin{lem}[\cite{bourgtelesc}, \cite{mullerdecomp}]\label{tele}Let $\lambda_1,\ldots,\lambda_n$, $\varphi_1,\ldots,\varphi_n$ be nonnegative random variables such that  \begin{equation}\label{eq:teleass}\E\lambda_k\geq \E\sqrt{\varphi_k^2+\lambda_{k-1}^2}.\end{equation} Then \begin{equation}\E\sqrt{\sum_{k=1}^n\varphi_k^2}\leq 2\sqrt{\E\lambda_n\E\max_{1\leq k\leq n}\lambda_k}.\end{equation}\end{lem}
\begin{cor}Let $X_1,\ldots,X_n$ be independent and set \begin{equation}\label{eq:lambdaconstdef}\lambda_0=0,\quad\lambda_k=\E\sqrt{X_k^2+\lambda_{k-1}^2}.\end{equation} Then \begin{equation} \lambda_n\leq \E\sqrt{\sum_{k=1}^nX_k^2}\leq 2\lambda_n.\label{oldcomp}\end{equation}\end{cor}
\begin{proof} The right inequality of \eqref{oldcomp} follows directly from Lemma \ref{tele}, since $\lambda_k$ is an increasing sequence of constants. To prove the other inequality, we see that conditioning with respect to $\sigma\left(X_{k+1},\ldots,X_n\right)$ gives \begin{eqnarray}
\E\sqrt{\sum_{j=k}^nX_j^2+\lambda_{k-1}^2} &=& \E\sqrt{\sum_{j=k+1}^n X_j^2 + \left(\sqrt{X_k^2+\lambda_{k-1}^2}\right)^2}\\ &\geq&  \E\sqrt{\sum_{j=k+1}^n X_j^2 + \lambda_k^2},
\end{eqnarray}
thus by induction \begin{equation}  \E\sqrt{\sum_{k=1}^nX_k^2}\geq \E\sqrt{\sum_{j=k+1}^n X_j^2 + \lambda_k^2},\end{equation} which for $k=n$ is the desired inequality.\end{proof}
\begin{proof}[Proof of Theorem \ref{tildemain}]
In order to prove the $\gtrsim$ inequality in \eqref{eq:tildesim}, we merely perform a slight modification of the proof of Lepingle inequality presented in \cite{bourgtelesc}. Let us denote $\Omega^N\times \Omega^N\ni (x,y)\mapsto f_n\p{x_1,\ldots,x_{n-1},n_k}$ by $\tilde{f_n}$. By tensoring $\left(f_n\right)$ against the Rademacher sequence, we may assume that it is a martingale difference sequence. Then the left hand side equals $\|F\|_{H^1}$, where $F=\sum_{n=1}^\infty f_n$ and $f_n=\Delta_n F$. By Theorem \ref{martatdec} it is enough to check the boundedness of the right hand side in the case when $F$ is an atom, because we have an a priori bound for finite sums. Let $F=u-\E_{j-1}u$, where $u$ satisfies \eqref{eq:atdef}. Then \beq\label{eq:fkdelta} f_k=\left\{\begin{array}{lll}0&\text{ if }&k<j\\ \Delta_k u&\text{ if }&k\geq j.\end{array}\right.\eeq 
By $A\in\mathcal{F}_j$, the support of $\E_k u $ for $k\geq j$ is contained in $A$ as well, because \begin{eqnarray}\E\left(\left|\E_k u\right|\cdot \mathbbm{1}_{\Omega^\infty\setminus A}\right)&\leq& \E\left(\E_k \left|u\right|\cdot \mathbbm{1}_{\Omega^\infty\setminus A}\right)\\ &=& \E\left(| u|\cdot  \E_k\mathbbm{1}_{\Omega^\infty\setminus A}\right)\\ &=&0.\end{eqnarray} Thus for $k>j$ we have $\supp f_k\subset \supp \E_k u\cup \supp \E_{k-1}u\subset A$. Consequently \beq \label{eq:fksupp}\supp \tilde{f}_k\subset A\times \Omega^\infty\text{ for }k>j,\eeq because if $\left(x,y\right)\in\supp \tilde{f}_k$, then $\left\{\left(x_1,\ldots,x_{k-1},y_k\right)\right\}\times \Omega^{\left\{k+1,\ldots\right\}}\subset \supp f_k\subset A$, which by $A\in\mathcal{F}_j$ implies $x\in A$. By and \eqref{eq:fkdelta} we have \beq \label{eq:fjl1}\E\left|\tilde{f}_j\right|= \E\left|f_j\right|=\E\left|\Delta_j u\right|
\leq 2\E|u|.\eeq Combining \eqref{eq:atdef}, \eqref{eq:fkdelta}, \eqref{eq:fksupp}, \eqref{eq:fjl1} with the inequality \beq \|X\|_{L^1}\leq \left|\supp X\right|^\frac12 \|X\|_{L^2}\eeq and the fact that the projections $\Delta_k$ are mutually orthogonal, we obtain
\begin{eqnarray}  \E\sqrt{\sum_{k=1}^\infty \tilde{f}_k^2}&\leq& \E\left|\tilde{f}_j\right|+  \E\sqrt{\sum_{k=j+1}^\infty \tilde{f}_k^2}\\ &\leq& 2\E\left|u\right|+ \E\sqrt{\sum_{k=j+1}^\infty \tilde{f}_k^2}\\ &\leq & 2|A|^\frac12 \left(\E u^2\right)^\frac12 +|A|^\frac12 \sqrt{\sum_{k=j+1}^\infty \E\tilde{f}_k^2}\\ &=& 2|A|^\frac12 \left(\E u^2\right)^\frac12 +|A|^\frac12 \sqrt{\sum_{k=j+1}^\infty \E\left(\Delta_k u\right)^2}\\ &\leq& 3|A|^\frac12 \left(\E u^2\right)^\frac12\\ &\leq & 3.\end{eqnarray}\par
We will prove the $\lesssim$ inequality in \eqref{eq:tildesim} now. It is clear that it is enough to prove it with only finitely many of $f_k$ nonzero. We define the sequence $\left(\lambda_k(x)\right)_{k=1}^n$ of functions in $L^1\left(\Omega^\infty\right)$ inductively by \beq \lambda_0(x)=0,\quad \lambda_k(x)=\int_{\Omega}\sqrt{f_k^2\left(x_1,\ldots,x_{k-1},y_k\right)+\lambda_{k-1}^2(x)}\d y_k.\eeq For any fixed $x\in\Omega^\infty$, this sequence coincides with the sequence defined by \eqref{eq:lambdaconstdef} applied to the independent random variables $f_k\left(x_1,\ldots,x_{k-1},\cdot\right)\in L^1(\Omega)$, so Corollary \ref{oldcomp} yields the pointwise inequality \beq \label{eq:lambdapointwise}\int_{\Omega^\infty}\sqrt{\sum_{k=1}^n f_k^2\left(x_1,\ldots,x_{k-1},y_k\right)}\d y\geq \lambda_n(x).\eeq By induction it is obvious that $\lambda_k$ is $\mathcal{F}_{k-1}$-measurable. Thus \begin{eqnarray} \lambda_k\left(x\right)&=&\int_{\Omega}\sqrt{f_k^2\left(x_1,\ldots,x_{k-1},y_k\right)+\lambda_{k-1}^2\left(x_1,\ldots,x_{k-2}\right)}\d y_k\\ &=&\left(\E_{k-1}\sqrt{f_k^2+\lambda_{k-1}^2}\right)\left(x_1,\ldots,x_{k-1}\right).\end{eqnarray} In particular, $\lambda_k$ verify the condition \eqref{eq:teleass} with respect to $f_k$ and are pointwise increasing, so Lemma \ref{tele} gives \beq \label{eq:lambdaincr}\E\sqrt{\sum_{k=1}^nf_k^2}\leq 2\left(\E\lambda_n\E\max_{k\leq n}\lambda_k\right)^\frac12 = 2\E\lambda_n.\eeq Integrating \eqref{eq:lambdapointwise} with respect to $x$ and applying \eqref{eq:lambdaincr} we obtain \begin{eqnarray} \E\sqrt{\sum_{k=1}^n\tilde{f}_k^2}&=&\int_{\Omega^\infty}\int_{\Omega^\infty} \sqrt{\sum_{k=1}^n f_k^2\left(x_1,\ldots,x_{k-1},y_k\right)}\d y\d x\\ &\geq& \int_{\Omega^\infty}\lambda_{n}(x)\d x\\ &\geq& \frac{1}{2}  \E\sqrt{\sum_{k=1}^n{f}_k^2}.\end{eqnarray}\end{proof}

\begin{proof}[Yet another proof of Theorem \ref{tildemain}] Without loss of generality, we may assume that $\Omega$ has a structure of a compact abelian group with Haar measure, e.g. by embedding $\Omega$ in $\T$. Just like previously, we also may assume that $f_k$ is a $k$-th martingale difference and notice that the left hand side is just $\left\|\sum_k f_k\right\|_{H^1}$. For $\xi\in \Omega^N$ we define an operator $T_\xi$ by 
\beq T_\xi f\p{x}= \sum_k \Delta_k f\p{x_1,\ldots,x_{k-1},x_k+\xi_k}.\eeq
Since $\Delta_k T_\xi f$ is just a translation of $\Delta_k f$,
\beq \left\|\Delta_k T_\xi f\right\|_{L^\infty}= \left\|\Delta_k f\right\|_{L^\infty}.\eeq
For $k>n$, by translation in the variable $y_k$, 
\begin{eqnarray} \E_n \left|\Delta_k T_\xi f\right|^2(x)&=& \int_{\Omega^{>n}}\left|\Delta_k f\p{x_1,\ldots,x_n,y_{n+1},\ldots,y_{k-1},y_{k}+\xi_k}\right|^2\d y\\
& =& \int_{\Omega^{>n}}\left|\Delta_k f\p{x_1,\ldots,x_n,y_{n+1},\ldots,y_{k-1},y_{k}}\right|^2\d y\\
&=& \E_n \left|\Delta_k f\right|^2(x).\end{eqnarray}
Therefore 
\beq \left\|T_\xi f\right\|_{BMO}\lesssim \|f\|_{BMO}.\eeq
We have $T_\xi^*=T_{-\xi}$, because 
\begin{eqnarray} \left\langle T_\xi f,g\right\rangle &=& \E\sum_k \Delta_k T_\xi f \Delta_k g\\
&=& \sum_k \int_{\Omega^k} \Delta_k f\p{x_1,\ldots,x_{k-1},x_k+\xi_k}\Delta_k g\p{x_1,\ldots,x_k}\d x\\
&=& \sum_k \int_{\Omega^k} \Delta_k f\p{x_1,\ldots,x_{k-1},x_k}\Delta_k g\p{x_1,\ldots,x_{k-1}-\xi_k}\d x\\
&=& \left\langle f,T_{-\xi}g.\right\rangle
\end{eqnarray}
By Theorem \ref{feffduality}, $T_{\xi}$ are uniformly bounded on martingale $H^1$, so 
\beq \|T_{\xi} f\|_{H^1}\lesssim \left\|f\right\|_{H^1}= \left\|T_{-\xi}T_{\xi}f\right\|_{H^1}\lesssim \left\|T_{\xi}f\right\|_{H^1}\eeq
and thus 
\beq \|f\|_{H^1} \simeq \int_{\Omega^N} \left\|T_\xi f\right\|_{H^1}\d \xi.\eeq
Ultimately, by translating $\xi$ for fixed $x$,
\begin{eqnarray} \int_{\Omega^N} \sqrt{\sum_k \left|f_k \p{x_1,\ldots,x_k}\right|^2}\d x&=& \left\|\sum_k f_k\right\|_{H^1}\\
&\simeq & \int_{\Omega^N}\left\|T_{\xi}\sum_k f_k\right\|_{H^1}\d \xi\\
&=& \int_{\Omega^N}\int_{\Omega^N} \sqrt{\sum_k \left|f_k \p{x_1,\ldots,x_{k-1},x_k+\xi_k}\right|^2}\d \xi\d x\\
&=& \int_{\Omega^N}\int_{\Omega^N} \sqrt{\sum_k \left|f_k \p{x_1,\ldots,x_{k-1},\xi_k}\right|^2}\d \xi\d x.
\end{eqnarray}\end{proof}

\end{document}